\documentclass[12pt]{article}
\usepackage{epsfig}
\usepackage{a4wide}
\usepackage{latexsym}
\usepackage{amssymb}
\usepackage{subfigure}
\usepackage{amsmath}
\usepackage{amscd}
\usepackage{tabularx}

\usepackage{graphicx} 

\newtheorem{theorem}{Theorem}
\newtheorem{proposition}[theorem]{Proposition}
\newtheorem{definition}{Definition}

\newtheorem{corollary}[theorem]{Corollary}
\newtheorem{conjecture}{Conjecture}

\newtheorem{observation}[theorem]{Observation}

\newcommand{\ora}[1]{\overrightarrow{#1}}
\newcommand{\eye}{\mathit{eye}}
\newcommand{\elb}{\mathit{elb}}

\newenvironment{proof}{\par \noindent {\bf Proof}.\ }{\hfill$\Box$ \par
\vspace{11pt}}

\begin{document}

\title{{\bf Covering line graphs with equivalence
    relations}\thanks{The three authors were supported by the European
    project \textsc{ist fet Aeolus}, contract number IP-FP6-015964,
    and hosted by the Institute for Theoretical Computer Science in
    Charles University, Prague, Czech Republic while they were working
    on this problem.}}

\author{Louis Esperet\thanks{CNRS, Laboratoire G-SCOP, Grenoble,
    France.  E-mail: {\tt louis.esperet@g-scop.fr}.}\and John
  Gimbel\thanks{Mathematical Sciences, University of Alaska,
    Fairbanks, AK, USA. E-mail: {\tt ffjgg@uaf.edu}.}\and Andrew
  King\thanks{Dept. of Industrial Engineering and Operations Research,
    Columbia University, NYC, USA. E-mail: {\tt
      aking6@cs.mcgill.ca}.}}

\date{}
\maketitle

\begin{abstract}

  An equivalence graph is a disjoint union of cliques, and the
  equivalence number $\mathit{eq}(G)$ of a graph $G$ is the minimum
  number of equivalence subgraphs needed to cover the edges of $G$.
  We consider the equivalence number of a line graph, giving improved
  upper and lower bounds: $\frac 13 \log_2\log_2 \chi(G) <
  \mathit{eq}(L(G)) \leq 2\log_2\log_2 \chi(G) + 2$.  This disproves a
  recent conjecture that $\mathit{eq}(L(G))$ is at most three for
  triangle-free $G$; indeed it can be arbitrarily large.

  To bound $\mathit{eq}(L(G))$ we bound the closely-related invariant $\sigma(G)$, which is the minimum number of orientations of $G$ such that for any two edges $e,f$ incident to some vertex $v$, both $e$ and $f$ are oriented out of $v$ in some orientation.  When $G$ is triangle-free, $\sigma(G)=\mathit{eq}(L(G))$.  We prove that even when $G$ is triangle-free, it is NP-complete to decide whether or not $\sigma(G)\leq 3$.\\

\end{abstract}

\section{Introduction}

Given a binary relation $\sim$ over a set $A$, it is natural to
consider expressing $\sim$ as a union of $k$ transitive subrelations
for the smallest possible value of $k$.  If $\sim$ is reflexive and
symmetric, each subrelation is an equivalence relation and we can
restate the problem as a graph covering problem: We seek to cover the
edges of a graph $G$ with $k$ {\em equivalence subgraphs}, i.e.\
subgraphs each of which is a disjoint union of cliques.  This is an
{\em equivalence covering} of $G$.  The minimum $k$ for which this is
possible is the {\em equivalence number} of $G$, denoted
$\mathit{eq}(G)$.

The equivalence covering number was introduced by Duchet in 1979
\cite{Duc79}.  Not surprisingly, it is NP-complete to compute, even
for split graphs \cite{BK95}.  In \cite{Alo86}, Alon proved upper and
lower bounds for general graphs:

\begin{theorem}
  Let $G$ be a graph on $n$ vertices with minimum degree $\delta$, and
  let $cc(G)$ be the minimum number of cliques needed to cover the
  edges of $G$. Then
$$\log_2n - \log_2(n-\delta-1) \leq \mathit{eq}(G) \leq cc(G) \leq 2e^2(n-\delta)^2\log_e n.$$
\end{theorem}


Observe that if $G$ is triangle-free, then every equivalence subgraph
of $G$ is a matching.  It follows that in this case an equivalence
covering of $G$ is actually an edge coloring, and that
$\mathit{eq}(G)$ is equal to the chromatic index $\chi'(G)$.  Thus
equivalence coverings can also be thought of as a generalization of
edge colorings.  In fact, McClain \cite{McC08+} formulated them
seemingly independently of earlier work in precisely this context,
calling $\mathit{eq}(G)$ the {\em clique chromatic index} of $G$.

In this paper we address the problem, first studied by McClain, of
bounding the equivalence number of line graphs.  For a graph $G$, the
line graph $L(G)$ of $G$ has a vertex corresponding to each edge of
$G$, and two vertices of $L(G)$ are adjacent precisely if the two
corresponding edges of $G$ share an endpoint (i.e.\ are {\em
  incident})\footnote{We need only consider line graphs of simple
  graphs: If two vertices $u$ and $v$ of $G$ have the same closed
  neighborhood, it is easy to see that
  $\mathit{eq}(G)=\mathit{eq}(G-v)$.  Thus we can easily reduce the
  problem for line graphs of multigraphs.}. McClain proved that for a
graph $G$ on $n$ vertices, $\mathit{eq}(L(G))\leq 4\left\lceil
  \frac{\log_en}{\log_e 12}\right\rceil$, and asked if this bound
could be improved \cite{McC08}.  We will prove that
$$\tfrac 13\left(\lceil \log_2\log_2 \chi(G)\rceil + 1\right) \le \mathit{eq}(L(G)) \le 2\left( \lceil \log_2 \log_2  \chi(G) \rceil +1\right),$$
where $\chi(G)$ is the chromatic number of $G$.  We will actually
prove a slightly better (but more unwieldy) lower bound.  Since
triangle-free graphs can have arbitrarily high chromatic number, our
lower bound disproves a recent conjecture of McClain~\cite{McC08+}
stating that any triangle-free graph $G$ has $\mathit{eq}(L(G))\leq
3$.

In order to bound $\mathit{eq}(L(G))$ we consider a closely-related
invariant of $G$, namely $\sigma(G)$.  In the next section we
introduce $\sigma(G)$ and prove that it is close to
$\mathit{eq}(L(G))$.  In Section 3 we relate $\sigma(G)$ to two other
interesting invariants arising from orientations.  In Section 4 we
will briefly discuss tightness and complexity concerns, in particular
proving that it is NP-complete to decide whether or not
$\mathit{eq}(L(G))\leq 3$, even if $G$ is triangle-free.

\section{Covering incidence pairs with orientations}


Equivalence subgraphs of a line graph $L(G)$ are intimately related to
orientations of $G$.  We begin the section by explaining why this is
so.

For every vertex $v$ of $G$, there is a clique $C_v$ of $L(G)$
corresponding to those edges of $G$ incident to $v$.  Every vertex of
$L(G)$ is in exactly two of these cliques, since every edge of $G$ has
two endpoints.  This fact invites a natural mapping from the set of
orientations of $G$ to the set of equivalence subgraphs of $L(G)$.
Given an orientation $\overrightarrow G$ we define the clique
$\overrightarrow{C_v}$ of $L(G)$ as the set of vertices of $L(G)$
corresponding to the out-edges of $v$.  For $u,v\in V(G)$ the cliques
$\overrightarrow{C_u}$ and $\overrightarrow{C_v}$ are disjoint, so the
disjoint union of $\overrightarrow{C_v}$ for all $v\in V(G)$ is an
equivalence subgraph of $L(G)$ corresponding to the orientation
$\overrightarrow G$.  We call this equivalence subgraph of $L(G)$ the
{\em analogue} of $\overrightarrow G$.

Using this idea, we can construct an equivalence covering of $L(G)$
using orientations of $G$.  Let $\overrightarrow{G_1},
\overrightarrow{G_2}, \ldots, \overrightarrow{G_k}$ be a set of
orientations of $G$ with the following property: For every vertex $u$
of $G$ with neighbors $v$ and $w$, there is an $i$ such that
$\overrightarrow{uv}, \overrightarrow{uw}\in \ora{E}(\ora{G_i})$.  In
other words, for every $e,f\in E(G)$ sharing an endpoint $v$, some
orientation $\ora{G_i}$ directs both $e$ and $f$ out of $v$.  We call
such a set of orientations an {\em orientation covering} of $G$, and
accordingly define the {\em orientation covering number} of $G$,
denoted $\sigma(G)$, as the size of a minimum orientation covering.
Figure 1 shows an orientation covering of size three for $K_4$, along
with a corresponding equivalence covering of size three for $L(K_4)$.

\begin{figure}
\begin{centering}
\includegraphics[scale=0.9]{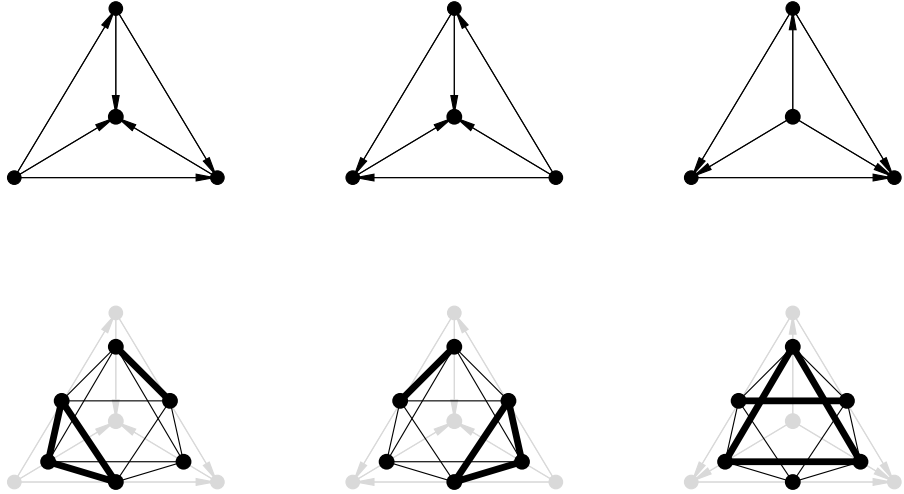}
\caption{An orientation covering of size three for $K_4$, and the
  corresponding equivalence subgraphs in $L(K_4)$.}\label{fig:k4}
\end{centering}
\end{figure}

Noting the discussion above, we can make an easy observation:

\begin{observation}\label{obs:obsleq}
For any graph $G$, we have $\mathit{eq}(L(G))\leq \sigma(G)$.
\end{observation}

The two invariants are actually equal for triangle-free $G$:

\begin{proposition}\label{prop:trifree}
For any triangle-free graph $G$, we have $\mathit{eq}(L(G))= \sigma(G)$.
\end{proposition}

\begin{proof}
  Let $H$ be an equivalence subgraph of $L(G)$, and consider a vertex
  $w$ of $L(G)$ corresponding to the edge $uv\in E(G)$.  Since $G$ is
  triangle-free, the neighborhood of $w$ in $H$ is contained in either
  $C_u$ or $C_v$.  We construct an orientation of $G$ from $H$ as
  follows: For every such $w$, $u$, and $v$, orient $uv$ towards $v$
  if $N_{H}(w)\subseteq C_u$, and orient it towards $u$ otherwise.
  For $w$ having no neighbor in $H$, orient $uv$ arbitrarily.

  If we construct an orientation of $G$ in this way for every
  equivalence subgraph of an equivalence covering of $L(G)$, it is
  easy to confirm that the result is an orientation covering of $G$.
  The result follows.
\end{proof}

The invariants $\sigma(G)$ and $\mathit{eq}(L(G))$ are not always
equal.  If $G$ is a triangle with a pendant vertex, then $\sigma(G)=3$
and $\mathit{eq}(L(G))=2$.  We suspect that this may be the worst
case, i.e.\ that $\sigma(G)\leq \mathit{eq}(L(G))+1$ for all connected
$G$, unless $G$ is a triangle, in which case $\mathit{eq}(L(G))=1$ and
$\sigma(G)=3$.  For now we simply show that they are within a
multiplicative constant of one another.

\begin{theorem}\label{th:lmlower}
For any graph $G$, $$\mathit{eq}(L(G))\leq \sigma(G) \le 3\,
\mathit{eq}(L(G)).$$
\end{theorem}

\begin{proof}
  Consider $k$ equivalence subgraphs $R_1, \ldots, R_k$ of $L(G)$
  covering all the edges of $L(G)$. Using
  Observation~\ref{obs:obsleq}, we only need to prove that $\sigma(G)
  \le 3k$. Take $i \in [k]$.  Each component of $R_i$ is either
  contained in $C_v$ for some $v\in V(G)$, or corresponds to the edges
  of a triangle in $G$.  Let $T_i$ be the set of triangles of $G$
  corresponding to cliques in $R_i$.  Observe that the triangles of
  $T_i$ must be edge-disjoint, since otherwise the corresponding
  triangles of $R_i$ would not be vertex-disjoint.  Consequently there
  exist three orientations $\overrightarrow{T_i}^1,
  \overrightarrow{T_i}^2, \overrightarrow{T_i}^3$ of the edges of
  $T_i$ such that for any triple $(u,v,w)$ of vertices of $G$
  corresponding to a triangle of $T_i$, the edges are oriented
  $\overrightarrow{uv}$ and $\overrightarrow{uw}$ in one of the
  orientations.

  We extend each $\overrightarrow{T_i}^j$ to an orientation
  $\ora{G_i}^j$ of $H$ as in the proof of Proposition
  \ref{prop:trifree}.  That is, for every $w\in V(R_i)$ corresponding
  to an edge $uv$ of $G$, we orient $uv$ towards $v$ if
  $N_{R_i}(w)\subseteq C_u$.  If $w$ has no neighbor in $R_i$, orient
  $uv$ arbitrarily. This construction gives us an orientation covering
  $\{\ora{G_i}^j \mid i\in [k],\ j\in [3] \}$, so $\sigma(G)\leq 3k$.
\end{proof}

This proves that $\mathit{eq}(L(G))$ and $\sigma(G)$ are within a
multiplicative constant of one another.  In the next section we prove
that $\sigma(G)$ is within a multiplicative constant of $\log_2 \log_2
\chi(G)$.

\section{Homomorphisms, eyebrows and elbows}


The bounds that we prove in this paper are generally stated in terms
of the chromatic number.  There is a simple justification for this,
which is that $\sigma(G)$ is monotone with respect to
homomorphism\footnote{A homomorphism from a graph $G$ to a graph $H$
  is a function $f:V(G)\rightarrow V(H)$ such that any two adjacent
  vertices of $G$ get mapped to adjacent vertices of $H$.}:

\begin{proposition}\label{prop:sigmahomo}
  Let $G$ and $H$ be graphs such that there is a homomorphism from $G$
  to $H$.  Then $\sigma(G) \leq \sigma(H)$.
\end{proposition}

\begin{proof}
  Consider a homomorphism $f: V(G)\rightarrow V(H)$ along with a
  minimum orientation covering of $H$.  For each orientation of $H$ we
  define an orientation of $G$ such that $\ora{uv}\in
  \ora{E}(\ora{G})$ precisely if $\ora{f(u)f(v)} \in
  \ora{E}(\ora{H})$.  It is straightforward to confirm that the
  resulting orientations of $G$ form an orientation covering, and we
  omit the details.
\end{proof}

\begin{corollary}\label{cor:chi}
For any graph $G$ with chromatic number $k$, $\sigma(G) \leq \sigma(K_k)$.
\end{corollary}

It is not clear whether or not there exists a graph $G$ for which
$\sigma(G)<\sigma(K_{\chi(G)})$.  However, this tightness does hold
for a related invariant which we now introduce.

Consider the following weakening of an orientation covering: Instead
of insisting that any two incident edges are out-oriented from their
shared endpoint in some orientation, we merely insist that in some
orientation they are either both out-oriented or both in-oriented.
This weakening inspires a new invariant.

\begin{definition}
  The elbow number $\elb(G)$ of a graph $G$ is the minimum $k$ for
  which there exist $k$ orientations $\{\ora{G_i} \mid i\in [k]\}$ of
  $G$ with the following property: For any path $u,v,w$ of $G$, there
  is an $i$ such that $u,v,w$ is not a directed path in $\ora{G_i}$.
  Such a collection of orientations is an {\em elbow covering}.
\end{definition}

Our interest in the elbow number comes primarily from two desirable
properties of the invariant.  First and foremost, it is not too far
from the orientation covering number:

\begin{proposition}\label{prop:sigmaelbow}
For any graph $G$, $\elb(G) \leq \sigma(G) \leq 2\elb(G)$.
\end{proposition}

\begin{proof}
  Clearly $\elb(G) \leq \sigma(G)$ because every orientation covering
  is also an elbow covering.  If we take a minimum elbow covering
  along with the reversal of each of its orientations, we get an
  orientation covering of size at most $2\elb(G)$.
\end{proof}

Second, a straightforward modification of the proof of Proposition
\ref{prop:sigmahomo} tells us that the elbow number is also monotone
under homomorphism:

\begin{proposition}\label{prop:elbowhomo}
  Let $G$ and $H$ be graphs such that there is a homomorphism from $G$
  to $H$.  Then $\elb(G) \leq \elb(H)$.  Consequently $\elb(G)\leq
  \elb(K_{\chi(G)})$.
\end{proposition}

We now characterize $\elb(G)$ precisely, beginning with the lower bound.

\begin{theorem}\label{thm:tightbound}
  For any graph $G$ with $\chi(G)\geq 3$, $\elb(G) \geq \lceil
  \log_2\log_2 \chi(G)\rceil+1$.
\end{theorem}

\begin{proof}
Suppose $\chi(G)\geq 3$ and $\elb(G)=k$.  Using a minimum elbow
covering of $G$, we will construct a proper coloring of $G$ using
$2^{2^{k-1}}$ colors.

Take an elbow covering of $G$ using $k$ orientations
$\ora{G_1},\ldots, \ora{G_k}$, and for every edge incidence $(u,uv)$,
let $o(u,uv)$ be the set of orientations for which $uv$ is oriented
out of $u$.  That is, $o(u,uv)= \{i \mid \ora{uv}\in
\ora{E}(\ora{G_i}) \}$.

The following properties of $o(u,uv)$ follow from the definition of an
elbow covering.  First, for adjacent $u$ and $v$, $o(u,uv) =
[k]\setminus o(v,uv)$.  Second, for $u$ with neighbors $v$ and $w$,
$o(u,uv) \neq [k]\setminus o(u,uw)$.  For if $o(u,uv)$ and $o(u,uw)$
partition $[k]$, then $v,u,w$ is a directed path in every orientation,
a contradiction.

For $X\subseteq [k]$, let $G_X$ be the subgraph of $G$ on those edges
$uv$ such that $o(u,uv)=X$ or $o(v,uv)=X$.  Note that the vertices of
$G_X=G_{[k]\setminus X}$ can be properly 2-colored with colors $X$ and
$[k] \setminus X$.  Therefore $G_X$ is bipartite.  Choose a 2-coloring
of every $G_X$; since $G_X=G_{[k]\setminus X}$ we can insist that
$G_X$ and $G_{[k]\setminus X}$ get the same 2-coloring.  Call this
2-coloring $c_X$, and observe that any two adjacent vertices $u$ and
$v$ get a different color in some $c_X$, namely $c_{o(u,uv)}$.  Thus
the product of $c_X$ over every possible $X\subseteq [k]$ gives us a
proper coloring of $G$.  Since $c_X=c_{[k]\setminus X}$, the coloring
uses $2^{2^{k-1}}$ colors.  Therefore $\chi(G)\leq 2^{2^{k-1}}$.
\end{proof}

Now we prove that the lower bound is tight.

\begin{theorem}\label{th:elbow}
  For any graph $G$ with $\chi(G)\geq 3$, $\elb(G) = \lceil
  \log_2\log_2 \chi(G)\rceil+1$.
\end{theorem}

\begin{proof}
  It suffices to show that $\elb(G) \leq \lceil \log_2\log_2
  \chi(G)\rceil+1$, and in particular it suffices to show this when
  $\chi(G) = 2^{2^\ell}$ for some nonnegative integer $\ell$.
  Proposition \ref{prop:elbowhomo} tells us that we can assume $G$ is
  the complete graph on $n=2^{2^\ell}$ vertices.  We proceed by
  induction.  If $\ell=1$ then $n=4$ and it is easy to confirm that
  $\elb(K_4)=2$.

  So assume $\elb(K_n)=k =\lceil \log_2\log_2 n\rceil+1$, let $G=K_n$,
  and let $\ora{G_1},\ldots, \ora{G_k}$ be a minimum elbow covering of
  $G$.  We will use this to construct an elbow covering of
  $G'=K_{n^2}$ as follows.  Label the vertices of $G$ as $\{v_i \mid 1
  \leq i \leq n \}$, and label the vertices of $G'$ as $\{v_i^j \mid 1
  \leq i,j \leq n \}$.  For each $\ora{G_i}$ we construct $\ora{G'_i}$
  such that $\ora{v_a^bv_c^d}\in \ora{E}(\ora{G'_i})$ precisely if
  $\ora{v_av_c}\in \ora{E}(\ora{G_i})$, or if $v_a=v_c$ and
  $\ora{v_bv_d}\in \ora{E}(\ora{G_i})$.  Finally, we add an
  orientation $\ora{G'_{k+1}}$ such that $\ora{v_a^bv_c^d}\in
  \ora{E}(\ora{G'_{k+1}})$ precisely if $\ora{v_av_c}\in
  \ora{E}(\ora{G_1})$, or if $v_a=v_c$ and $\ora{v_dv_b}\in
  \ora{E}(\ora{G_1})$.  In other words, we compose $\ora{G_i}$ with
  itself for each $i$, then we compose $\ora{G_1}$ with its reversal.

  Now consider the possibility that $v_i^j$, $v_a^b$, $v_c^d$ form a
  directed path in every orientation of $G'$.  By the construction of
  our orientations of $G'$, it is easy to see that $|\{i,a,c \}|= 2$.
  So assume without loss of generality that $i=a\neq c$.  Since the
  edge $v_i^jv_a^b$ will be oriented differently in $\ora{G'_1}$ and
  $\ora{G'_{k+1}}$ and the edge $v_a^bv_c^d$ will be oriented the
  same, it follows that $v_i^j$, $v_a^b$, $v_c^d$ cannot be a directed
  path in both orientations.  Therefore we have an elbow covering of
  $K_{n^2}$ of size $k+1$, and the theorem follows by induction.
\end{proof}

This gives us a bound on $\sigma(G)$:

\begin{corollary}\label{cor:tight}
  For any graph $G$ with $\chi(G)\geq 2$, $$\lceil \log_2\log_2
  \chi(G)\rceil+1 \ \leq \ \sigma(G) \ \leq \ 2\lceil\log_2\log_2
  \chi(G)\rceil+2.$$
\end{corollary}
\begin{proof}
  When $\chi(G)\geq 3$, this follows immediately from Proposition
  \ref{prop:sigmaelbow} and the previous theorem.  It is easy to see
  that any bipartite graph has orientation covering number at most
  two: if $V(G)$ is covered by two disjoint stable sets $A$ and $B$,
  we simply choose one orientation in which all vertices in $A$ are
  sources, and one orientation in which all vertices in $B$ are
  sources.  The result follows.
\end{proof}

We can actually improve the lower bound by exploiting properties of
orientation coverings to refine the proof of Theorem
\ref{thm:tightbound}:

\begin{theorem}\label{thm:tighter}
Any graph $G$ with $\sigma(G)=k\geq 3$ has $\chi(G)\leq k+2^{2^{k-1}-k-1}$.  Thus
$$k \geq \log_2\left(\log_2(\chi(G)-k)+k+1\right)
$$
\end{theorem}

\begin{proof}
  Let $G$ be a minimum counterexample.  We can assume $G$ has no
  vertex of degree 1, since removing such a vertex will change neither
  $\sigma(G)$ nor $\chi(G)$.  Consider an orientation covering
  $\ora{G_1},\ldots \ora{G_k}$ of $G$, and let $\ell =
  k+2^{2^{k-1}-k-1}$.  We will construct an $\ell$-coloring of $G$. As
  in the proof of Theorem~\ref{thm:tightbound}, we set $o(u,uv)= \{i
  \mid \ora{uv}\in \ora{E}(\ora{G_i}) \}$ for any incidence
  $(u,uv)$. First, for $i\in [k]$ let $S_i$ be the set of vertices $v$
  having a neighbor $u$ such that $o(v,uv)=\{i\}$.  Each $S_i$ is a
  stable set.  Let $S=\cup_{i}S_i$ and let $U=V(G)\setminus S$.  We
  now proceed to color $U$ using $\ell-k=2^{2^{k-1}-k-1}$ colors.

  We claim that for any adjacent vertices $u, v\in U$, $2\leq
  |o(v,uv)|\leq k-2$.  Clearly $o(v,uv)$ cannot be empty or equal to
  $[k]$ by properties of an orientation covering, since $G$ has
  minimum degree at least two.  And $o(v,uv)$ cannot have size 1 or
  $k-1$, otherwise either $u$ or $v$ would be in $S$.  Thus there are
  $2^k-2k-2$ possibilities for $o(v,uv)$, and for each possibility we
  get a bipartite graph, as in the proof of Theorem
  \ref{thm:tightbound}.  And again as in the proof of Theorem
  \ref{thm:tightbound}, we actually get a bipartite graph for every
  complementary pair of subsets of $[k]$.  Thus we color $U$ by taking
  the product of 2-colorings of $2^{k-1}-k-1$ bipartite subgraphs.
  This gives us an $(\ell-k)$-coloring of $U$ and an $\ell$-coloring
  of $G$.
\end{proof}

Although this bound on $\sigma(G)$ may seem ungainly, we will see in
Section~\ref{sec:tight} that it is tight for small values of
$\chi(G)$.

Just as we have bounded $\sigma(G)$ and $\elb(G)$ in terms of
$\chi(G)$, we can bound $\chi(G)$ in terms of $\sigma(G)$ and
$\elb(G)$.

\begin{corollary}\label{cor:converse}
For any graph $G$ with $\elb(G)\geq 2$ and $\sigma(G)\geq 3$, 
$$2^{2^{\elb(G)-2}} \ < \ \chi(G) \  \leq   \ 2^{2^{\elb(G)-1}}$$
and
$$2^{2^{(\sigma(G)/2)-2}} \ < \ \chi(G) \  \leq   \ \sigma(G)+2^{(2^{(\sigma(G)-1)}-\sigma(G)-1)}.$$
\end{corollary}

\subsection{Elbows versus eyebrows}\label{sec:eye}

The elbow number of a graph is very closely related to the {\em
  eyebrow number} of a graph, studied by K\v r\' i\v z and Ne\v set\v
ril \cite{kriznesetril91} and defined thusly:

\begin{definition}
  The eyebrow number $\eye_\pi(G)$ of a graph $G$ is the minimum $k$
  for which there exist $k$ permutations $\{\pi_i \mid i\in [k]\}$ on
  $V(G)$ with the following property: For any edge $uv$ of $G$ and
  third vertex $w$, there is an $i$ such that $\pi_i(w)$ is not
  between $\pi_i(u)$ and $\pi_i(v)$.
\end{definition}

The connection between permutations and acyclic orientations is the
following. For a permutation $\pi$ of $[n]$ and a graph $G$ with
vertex set $\{v_1,\ldots,v_n\}$, define the following acyclic
orientation $\ora{G_{\pi}}$ of $G$: an edge $v_iv_j$ of $G$ is
oriented from $v_i$ to $v_j$ in $\ora{G_{\pi}}$ precisely if $\pi(i) <
\pi(j)$. Conversely, an acyclic orientation of $G$ is a partial order
on its vertices, and any linear extension of this order corresponds to
a permutation of $[n]$.  In~\cite{kriznesetril91}, K\v r\' i\v z and
Ne\v set\v ril proved that $\eye_\pi(K_n)= \lceil\log_2\log_2
n\rceil+1$ using the tightness of a result of Erd\H os and
Szekeres~\cite{ES35} that is closely related to our proof of the upper
bound on $\elb(G)$.  This can be used to provide an alternative proof
of the upper bound in Corollary \ref{cor:tight}.  Like us, they were
motivated by a different problem.  They were interested in proving the
existence of posets of bounded dimension whose Hasse diagrams could
have arbitrarily high chromatic number.  It follows immediately from
Theorem \ref{th:elbow} that for a complete graph, the eyebrow number
and elbow number are equal.  However, the eyebrow number is not
monotonic under homomorphism -- there are graphs for which
$\eye_\pi(G) = \eye_\pi(K_{\chi(G)})+1 = \elb(K_{\chi(G)})+1 =
\elb(G)+1$ (for example, a sufficiently large complete and regular
tripartite graph \cite{kriznesetril91}).  More fundamentally, the
eyebrow number of a graph does not really reflect the structure of
incident edge pairs.  This is the first reason behind our interest in
the elbow number as opposed to the eyebrow number.

The second reason is that we do not want to restrict ourselves to
acyclic orientations. McClain \cite{McC08+} asked whether, for any
$n$, $L(K_n)$ has a minimum equivalence covering in which every
equivalence subgraph is the analogue of an acyclic orientation of
$K_n$ (recall we defined an analogue in Section 2).  Theorem
\ref{thm:tightbound} answers the corresponding question in the
affirmative for elbow coverings.  That is, using orientations with
cycles does not help in constructing a minimum elbow covering.
However the question remains open for the orientation covering number.

\section{Tightness and complexity}\label{sec:tight}

Let us consider our upper bound on $\sigma(G)$, which we got from the
bound on $\elb(G)$.  Corollary \ref{cor:tight} implies that
$\sigma(K_{16})\leq 6$, but as one might expect, this bound is not
tight. The five (acyclic) orientations of $K_{16}$ associated to the
following five permutations show that $\sigma(K_{16})\le 5$:

\vspace{8pt}

\begin{center}
\begin{small}
\begin{tabular}{|c||c|c|c|c|c|c|c|c|c|c|c|c|c|c|c|c|}
\hline $i$ & 1 & 2 & 3 & 4 & 5 & 6 & 7 & 8 & 9 & 10 & 11 & 12
 & 13 & 14 & 15 & 16 \\
\hline \hline $\pi_1(i)$ & 1 & 2 & 3 & 4 & 5 & 6 & 7 & 8 & 9 & 10 & 11 & 12
 & 13 & 14 & 15 & 16\\

\hline $\pi_2(i)$ &  13 & 11 & 10 &  6 &  4 &  9 &  5 &  3 &  7 &  2 & 12 &
8 & 14 & 15 & 16 &  1\\
\hline $\pi_3(i)$ &   14 & 11 & 10 &  3 &  8 & 12 &  5 &  7 &  2 &  9 &  4 &
6 & 15 & 16 &  1 & 13\\
\hline $\pi_4(i)$ &  15 &  7 &  8 &  9 &  6 &  4 &  3 & 12 & 10 & 11 &  5 &
2 & 16 &  1 & 13 & 14\\
\hline $\pi_5(i)$ & 16 &  5 &  4 & 10 & 11 &  3 & 12 &  6 &  9 &  7 &  2 &  8
&  1 & 13 & 14 & 15\\
\hline

\end{tabular}
\end{small}
\end{center}

\vspace{5pt}

Corollary \ref{cor:converse} tells us that if $\sigma(G)=3$, then $3
\leq \chi(G)\leq 4$ -- the lower bound comes from the easy fact that
$\sigma(G)\le 2$ precisely if $G$ is bipartite.  The converse is also
true by Corollary~\ref{cor:chi} and the orientation covering of size
three of $K_4$ depicted in Figure \ref{fig:k4}. If $\sigma(G)=4$ then
Corollary \ref{cor:converse} tells us that $\chi(G)\leq 12$.  This is
tight as well -- an example due to McClain \cite{McC08+} implies that
$\sigma(K_{12})= 4$.  So there is some evidence that the improved
bound of Theorem \ref{thm:tighter} may be tight or nearly tight in
general. As a consequence of these observations, we obtain the
following two equivalences:

\begin{theorem}\label{th:fineq}
  A graph $G$ has $\sigma(G)=3$ precisely if $3 \leq \chi(G)\leq 4$,
  and $\sigma(G)=4$ precisely if $5 \leq \chi(G)\leq 12$.
\end{theorem}

Blokhuis and Kloks \cite{BK95} proved that $\mathit{eq}(G)$ is
NP-complete to compute, even if it is at most four and $G$ has maximum
degree at most six and clique number at most three.  As proved by
Maffray and Preissmann \cite{MP96}, it is NP-complete to decide
whether or not $G$ is $k$-colorable for $k\geq 3$, even when $G$ is
triangle-free.  As a consequence, $\sigma(G)$ is difficult to compute,
as is $\mathit{eq}(L(G))$:

\begin{theorem}
  It is NP-complete to decide whether or not a triangle-free graph $G$
  has $\sigma(G) \leq 3$ (resp. $\sigma(G) \leq 4$). Equivalently, it
  is NP-complete to decide whether or not $\mathit{eq}(L(G)) \leq 3$
  (resp. $\mathit{eq}(L(G)) \leq 4$).
\end{theorem}

In fact, we conjecture that this also holds for all larger values of
$\sigma$:

\begin{conjecture}
For any $k\ge 3$, it is NP-complete to decide whether or not $\sigma(G) \le k$.
\end{conjecture}

\section{Conclusion}

Theorem~\ref{th:lmlower} implies that for any graph $G$, $\tfrac13 \,
\sigma(G) \le \mathit{eq}(L(G)) \le \sigma(G)$. Applying
Proposition~\ref{prop:sigmaelbow}, we obtain $\tfrac13\, \elb(G) \le
\mathit{eq}(L(G)) \le 2\,\elb(G)$. If $\chi(G) \ge 3$,
Theorem~\ref{th:elbow} tells us that $$\tfrac 13\, (\lceil
\log_2\log_2\chi(G)\rceil+1)\, \le\, \mathit{eq}(L(G))\, \le\, 2\,
(\lceil \log_2\log_2\chi(G)\rceil+1).$$ As a consequence,
$\mathit{eq}(L(G))$ is unbounded, answering a question of
\cite{McC08}.  Further, as the chromatic number is unbounded for
triangle-free graphs, $\mathit{eq}(L(G))$ is not bounded above by
three; this disproves a conjecture in \cite{McC08+}. The tighter
Theorem \ref{thm:tighter} implies
that $$\log_2\left(\log_2\left(\chi(G)-3\mathit{eq}(L(G))\right)+3\mathit{eq}(L(G))+1\right)\leq
\mathit{eq}(L(G)).$$

Otherwise if $G$ is bipartite, then $\mathit{eq}(L(G)) = \sigma(G)=2$,
so the the previous inequalities still hold. Finally if $\chi(G)=1$,
the graph $L(G)$ has no vertices and the inequalities are meaningless.

There are several compelling problems that remain to be solved.  First
is an improved bound on $\sigma(K_n)$.  We believe that it is closer
to the lower bound than the upper bound, and we even think that the
lower bound might be tight.  The second question is that of bounding
$\sigma(G)$ in terms of $\mathit{eq}(L(G))$.  We suspect that they
differ by at most an additive constant for any graph.  Finally, we
would like to know if there is some graph $G$ for which
$\sigma(G)<\sigma(K_{\chi(G)})$.

\paragraph{Additional remarks}

After this draft was submitted, Andr\'as Gy\'arf\'as remarked that
$\chi(DS_n) \le \sigma(K_n) \le \chi(DS_n)+2$, where $\chi$ is the
chromatic number and $DS_n$ is the double-shift graph on $n$
vertices. In~\cite{FHRT92}, it was proved that $\chi(DS_n)=\log_2
\log_2 n + (1/2+o(1))  \log_2 \log_2 \log_2 n$, which directly improves
the upper bound in Corollary~\ref{cor:tight}.

We then realized that using~\cite[Theorem 3]{Spe72}, we can prove the
following even stronger statement: $L(K_n)$ can be covered with
$\log_2 \log_2 n + (1/2+o(1)) \log_2 \log_2 \log_2 n$
equivalence subgraphs, each of which is the analogue of an acyclic
orientation of $K_n$. This indicates that the last question of
Section~\ref{sec:eye} might have a positive answer.

\end{document}